\documentclass[a4paper,11pt,reqno,oneside]{article}
\usepackage{graphicx} 
\usepackage{amsmath}
\usepackage[utf8]{inputenc}
\usepackage{amssymb}
\usepackage{amsthm}
\usepackage{mathtools}
\usepackage{graphicx}
\usepackage{color, soul}
\usepackage{centernot}
\usepackage{verbatim}
\usepackage{multirow}
\usepackage{array}   
\newcolumntype{L}{>{$}l<{$}} 
\usepackage{setspace}
\usepackage{mathrsfs}
\usepackage{enumitem}
\usepackage{tikz}
\usepackage{bigints}
\usepackage{bm}
\usepackage{caption}
\usepackage{subcaption}
\usepackage[all,cmtip]{xy}

\usepackage{url}
\usepackage{imakeidx}
\makeindex[columns=2, intoc]
\usepackage{hyperref}
\usepackage{tikz-cd}

\usepackage{natbib}

\usepackage{titling}

\makeatletter
\renewcommand{\@biblabel}[1]{[#1]\hfill}
\makeatother

\newtheorem{theorem}{Theorem}[section]
\newtheorem{corollary}[theorem]{Corollary}
\newtheorem{lemma}[theorem]{Lemma}
\newtheorem{proposition}[theorem]{Proposition}
\newtheorem{conj}[theorem]{Conjecture}

\theoremstyle{definition}

\theoremstyle{remark}

\newtheorem{remarkx}[theorem]{Remark}
\newenvironment{remark}
  {\pushQED{\qed}\remarkx}
  {\popQED\endremarkx}

\numberwithin{equation}{section}

\usepackage{mathtools}

\let\temp\phi
\let\phi\varphi
\let\varphi\temp

\newcommand{\oo}{\mathcal{O}}

\newcommand{\C}{\mathbb{C}}

\newcommand{\Q}{\mathbb{Q}}
\newcommand{\Z}{\mathbb{Z}}

\newcommand{\G}{\mathbb{G}}

\newcommand{\Proj}{\mathbb{P}}

\DeclareMathOperator{\et}{\acute{e}t}
\DeclareMathOperator{\br}{Br}
\DeclareMathOperator{\h}{H}

\DeclareMathOperator{\NS}{NS}

\DeclareMathOperator{\Gal}{Gal}

\DeclareMathOperator{\Spec}{Spec}

\DeclareMathOperator{\Hom}{Hom}

\DeclareMathOperator{\Pic}{Pic}

\DeclareMathOperator{\End}{End}
\DeclareMathOperator{\Jac}{Jac}
\DeclareMathOperator{\Aut}{Aut}

\DeclareMathOperator{\rk}{rk}

\DeclareFontFamily{U}{wncy}{}
\DeclareFontShape{U}{wncy}{m}{n}{<->wncyr10}{}
\DeclareSymbolFont{mcy}{U}{wncy}{m}{n}
\DeclareMathSymbol{\Sha}{\mathord}{mcy}{"58}

\newcommand{\pp}{\mathfrak{p}}

\newcommand{\qq}{\mathfrak{q}}

\newcommand{\legendre}[2]{\ensuremath{\left( \frac{#1}{#2} \right) }}

\renewcommand{\bar}{\overline}

\usepackage{geometry}
\usepackage{fancyhdr}

\title{Transcendental Brauer groups of cubic generalised Kummer surfaces}
\author{Giorgio Navone}
\date{}

\begin{document}

\maketitle
\begin{abstract}
Given a cubic curve $C$ over a number field, we consider the K3 surface $Y_C$ constructed as the minimal desingularisation of the quotient of $C \times C$ by an automorphism of order 3. We relate the transcendental Brauer groups of $Y_C$ and $C \times C$, allowing us to explicitly compute the former group in the case of a diagonal cubic curve defined over $\Q$. We obtain conjectural insight on the existence of Galois cubic points over $\Q$ for everywhere locally soluble diagonal cubic curves.
\end{abstract}

\section{Introduction}

A natural approach to determine the existence or non-existence of rational points on a variety $X/\Q$ is to first check the existence of $\Q_p$-rational points for all $p \leq \infty$. Consider a family of varieties $\{X_i \}_{i \in I}$. If, for all $X_i$, we have $X_i(\Q) \neq \emptyset$ if and only if $X_i(\Q_p) \neq \emptyset$ for all $p \leq \infty$, then we say that the family $\{X_i \}$ satisfies the Hasse principle. This is indeed the case for quadrics and Severi--Brauer varieties, for instance, but in general the Hasse principle may fail. The failure is often explained by the Brauer--Manin obstruction; see Chapter 13 in $\cite{colliot2021brauer}$ for a precise formulation. The idea is that inside the set of adelic points $X(\mathbb{A}_{\Q})$ lies the so-called Brauer--Manin set $X(\mathbb{A}_{\Q})^{\br} \subset X(\mathbb{A}_{\Q})$, consisting of those adelic points that satisfy certain conditions imposed by the Brauer group. Class field theory shows that these conditions are always satisfied by the rational points $X(\Q)$. The key point is that emptiness of the Brauer--Manin set can give an obstruction to the existence of rational points, even when $X(\Q_p) \neq \emptyset$ for all $p\leq \infty$.

This approach to rational points justifies the interest in studying the Brauer groups of varieties, with the goal of computing them as abstract groups and describing their elements in a way amenable to the obstruction computation. The Brauer group $\br(X)$ is endowed with a natural filtration $\br_0(X) \subset \br_1(X) \subset \br(X)$, defined as  
\[\br_0(X)=\mathrm{Im}(\br(\Q) \to \br(X))  \subset  \br_1(X)=\mathrm{ker}(\br(X) \to \br(\bar{X}) ). \]
The elements in $\br_0(X)$ are called \emph{constant} elements and these will be discarded since they never yield obstructions to rational points. The \emph{algebraic} Brauer group $\br_1(X)$ can play a role in the obstruction and it is usually computed using the Hochschild--Serre spectral sequence, which provides the isomorphism $\br_1(X)/\br_0(X) \cong \h^1(\Q, \Pic(\bar{X}))$. 
Finally, the elements of $\br(X)/\br_1(X)$ are called \emph{transcendental} elements and they are troublesome to study, since knowledge of $\br(\bar{X})^{\Gamma_{\Q}}$ is not in general sufficient to determine which elements descend to the base field $\Q$.

The Brauer--Manin obstruction is not the only obstruction to the Hasse principle, as first shown by Skorobogatov in his famous paper $\cite{skorobogatov1999manin}$. However, it is conjectured to be the only obstruction for some important families of varieties, such as curves (\cite{Skorobogatov_2001}), geometrically rationally connected varieties (\cite{CTconjecture2003}), and K3 surfaces (\cite{skorobogatov_2009}).  

The connection between certain K3 surfaces and abelian surfaces or elliptic curves has been exploited to study the Brauer group and the corresponding obstruction; see the results on the transcendental Brauer groups of Kummer surfaces in $\cite{skorobogatov2010brauer}$ and $\cite{skorobogatovzarhin2017}$ or the results on diagonal quartic surfaces in $\cite{diagonalquartic2011}$ and \cite{diagonalquartic2015}, for example. 

This paper extends work of Van Luijk on the Brauer group of a family of generalised Kummer surfaces $Y_C$, defined as the minimal desingularisation of the quotient of $C \times C$ by an automorphism of order 3, where $C$ is a diagonal cubic curve defined by $C: ax^3+by^3+cz^3=0$, over a number field $k$. Indeed, \cite[Prop~4.2]{vanluijk2007cubic} shows that the algebraic Brauer groups of these surfaces contain only constant elements, except potentially when $abc$ is a cube in $k$.

Our first main theorem investigates the relationship between the transcendental Brauer groups of $Y_C$ and $C \times C$, obtaining a result analogous to Theorem 2.4 in \cite{skorobogatov2010brauer}. 
\begin{theorem}[Theorem \ref{main_thm}] \label{theorem1.1}
  Let $C$ be a smooth projective cubic curve defined over a number field. There is an embedding 
    \[ \frac{\br(Y_C)_n}{\br_1(Y_C)_n} \hookrightarrow \frac{\br(C \times C)_n}{\br_1(C \times C)_n},  \]
    which is an isomorphism if $n$ is coprime to 3. The subgroups of elements of order coprime to 3 of the transcendental Brauer groups $\br(Y)/\br_1(Y)$ and $\br(C \times C)/\br_1(C \times C)$ are isomorphic.   
\end{theorem}
This result allows us to explicitly compute the Brauer group of $Y_C$ over $\Q$, using the description of the transcendental Brauer group of products of CM elliptic curves given by Newton in \cite{rachel2015}.     
\begin{theorem}[Corollary \ref{main corollary}, Theorem \ref{NoObstructionThm}] \label{theorem1.2}
    Let $C/\Q$ be a diagonal cubic curve $ax^3+by^3+cz^3=0$  such that $abc$ is not a cube in $\Q$. Then,
   \[\frac{\br(Y_C)}{\br_0(Y_C)}= \begin{cases}
    \Z/2\Z &\text{ if } 4abc \text{ is a  cube in } \Q; \\
   0 &\text{ otherwise.}
\end{cases} \]
Moreover, there is no Brauer--Manin obstruction to the Hasse principle for $Y_C$.
\end{theorem}
Theorem \ref{theorem1.2} builds upon \cite[Prop~4.2]{vanluijk2007cubic}, wherein Van Luijk showed that $\br_1(Y_C)=\br_0(Y_C)$.
Skorobogatov has conjectured that the Brauer--Manin obstruction is the only obstruction to the Hasse principle for K3 surfaces. Assuming this conjecture and everywhere local solubility of $C/\Q$, Theorem \ref{theorem1.2} proves the existence of \emph{Galois cubic points} on $C$, i.e.\ $L$-rational points on $C$ for a Galois extension $L/\Q$ of degree 3. 
\begin{theorem}[Theorem \ref{galoiscubicpoints}] \label{theorem1.3}
  Let $C/\Q$ be an everywhere locally soluble smooth curve with equation $ax^3+by^3+cz^3=0$. Then, Skorobogatov's conjecture implies the existence of Galois cubic points on $C$.  
\end{theorem}
The proof of Theorem \ref{theorem1.3} gives a strengthening of \cite[Thm~1.1]{vanluijk2007cubic} in the case $k = \Q$; it shows that one can replace the algebraic Brauer group in the statement of \cite[Thm~1.1]{vanluijk2007cubic} with the full Brauer group of $Y_C$.  
Thus, it shows that Skorobogatov's conjecture implies a negative answer to \cite[Question 2]{vanluijk2007cubic} over $\Q$.
The author would be very interested in unconditionally proving the existence of Galois cubic points on $C$ as in Theorem \ref{theorem1.3}, or relating this problem to the Brauer--Manin obstruction for curves instead. 
\subsection*{Outline of the paper} After recalling some technicalities in Section 2, Section 3 is devoted to the proof of Theorem \ref{theorem1.1}. In Section 4 we prove Theorem \ref{theorem1.2}. Section 5 briefly discusses the Hasse principle for the surface $Y_C$ and the implications for Galois cubic points on $C$.      
\subsection*{Notation} We fix some notation. 
\begin{itemize} 
    \item $k$ is a number field and $\Gamma_k$ is the absolute Galois group $\Gal(\bar{k}/k)$;
    \item for a $k$-variety $X$, we denote by $\bar{X}$ the base change to an algebraic closure $X \times_k \bar{k}$; for $x \in \br(X)$, we denote by $\bar{x}$ the image of $x$ via the natural map $\br(X) \to \br(\bar{X})$;
    \item for an abelian group $G$ and a positive integer $n$, we denote by $G_n$ the subgroup of elements of order dividing $n$; for a $G$-module $M$, we denote by $M^G$ the submodule of $G$-invariants;
    \item if $f:A \to B$ is an isogeny of abelian varieties, we denote by $f^{\vee}: B^{\vee} \to A^{\vee}$ the dual isogeny.
\end{itemize}

\section{Preliminaries} 

In this section we gather the most important technical tools (well known in the literature) that are required for the proofs in Section 3. We begin with the statement, and sketch of the proof, of a very general lemma. 

\begin{lemma}{\label{general_lemma}}
    Let $X$ be a smooth and geometrically integral variety over $k$, let $U$ be an open subset whose complement in $X$ has codimension $\geq 2$. Then the following hold:
    \[ k[U]=k[X], \qquad \Pic(U)=\Pic(X), \qquad \br(U)=\br(X). \]
\end{lemma}
\begin{proof}
    The first equality comes from the fact that $U$ is dense in $X$ because $X$ is irreducible. Since codimension of $X \setminus U$ is at least two, the restriction gives a correspondence between (Weil-)divisors of $X$ and $U$. Finally, the result about the Brauer groups derives directly from Grothendieck's purity theorem. 
\end{proof}

Let $\rho_X$ be the rank of the Néron--Severi group $\NS(\Bar{X})$ and let $b_2$ be the second Betti number for a smooth, projective, geometrically integral variety $X$. Then for every prime number $l$, we have the following exact sequence (\cite[Prop.~5.2.9]{colliot2021brauer})
\[0 \to (\Q_l/\Z_l)^{b_2-\rho_X} \to \br(\Bar{X})\{ l\} \to \h^3_{\et}(\Bar{X},\Z_l(1))\to 0\]
Moreover, since $k \subset \C$, we have that $\bigoplus_{l} \h^3_{\et}(\Bar{X},\Z_l(1))$ is isomorphic to the torsion subgroup of $\h^3(X(\C),\Z)$. If $X$ is a surface such that $\NS(\Bar{X})$ has no $l$-torsion, then $\h^1(X(\C),\Z)$ has no $l$-torsion, thus by Poincaré duality $\br(\Bar{X})\{l\} \cong (\Q_l/\Z_l)^{b_2-\rho_X}$. 
Summing up, since the Brauer group is torsion, it follows that if $X$ is an abelian variety or a K3 surface we have $\br(\Bar{X}) \cong (\Q/\Z)^{b_2-\rho_X}$.

Finally, a crucial proposition relating the Brauer groups for a Galois covering of smooth varieties. 
\begin{proposition}\textnormal{\cite[3.8.5]{colliot2021brauer} }{\label{galois_covering}}
Let $X$ and $Y$ be smooth varieties and let $f : Y \to X$ be a finite flat morphism of degree $d$ such that $k(Y)$ is a Galois extension of $k(X)$ with Galois group $G$. Then, for any $n$ coprime to $d=\lvert G \rvert$, the natural map $f^*: \br(X)_n \to \br(Y)^G_n$ is an isomorphism. 
\end{proposition}

\subsection{Néron--Severi groups of product abelian surfaces}

In this section we present some facts about the Néron--Severi groups of product abelian varieties; in particular  we are interested in the case of product abelian surfaces. We recommend \cite[pp.~189-190]{mumford_av} and the appendix in \cite{kani2016} for more details on this matter. 

Let $A$ be an abelian variety over an algebraically closed field and let $\NS(A) = \Pic(A)/\Pic^0(A)$ denote the Néron--Severi group of $A$, whose elements are classes of line bundles up to rational equivalence. For any line bundle $\mathcal{L}$, denote with $\lambda_{\mathcal{L}}$ the polarization given by $\mathcal{L}$, i.e.\ the map  
\begin{align*}
    \lambda_{\mathcal{L}}: A & \to \Pic(A) \\ a & \mapsto t^*_a \mathcal{L} \otimes \mathcal{L}^{-1}.  \end{align*} 
    We assume that $A$ is principally polarized with polarization $\lambda : A \xrightarrow{\sim} A^{\vee}$.  Using the Rosati involution $r_{\lambda}$ on $\End(A)$ defined by $r_{\lambda} (\alpha)= \lambda^{-1} \alpha^{\vee} \lambda$, the 
 group $\NS(A)$ can be described as a subgroup of $\End(A)$ in the following way.
\begin{proposition} \textnormal{(\cite{mumford_av})}
    The map $D \mapsto \lambda^{-1} \lambda_D$ defines a functorial isomorphism  \[ \Phi_{\lambda}: \NS(A) \xrightarrow{ \sim} \End_{\lambda}(A) :=\{ \alpha \in \End(A): r_{\lambda}(\alpha)=\alpha \}. \]
\end{proposition}
It is possible to be more explicit in the case of a product of two elliptic curves. Let $A=E_1 \times E_2$ and let $p_1: A \to E_1$, $p_2:A \to E_2$ be the projections and let $e_1: E_1 \to A$, $e_2:E_2 \to A$ be the natural embeddings. Then, $p:=p_1^{\vee}+p_2^{\vee} : E_1^{\vee} \times E_2^{\vee} \to A^{\vee}$ is an isomorphism and given two principal polarizations $\lambda_i$ of $E_i$ we can define a principal polarization on $A$ as $\lambda_1 \times \lambda_2 := p \circ (\lambda_1 \times \lambda_2) : A \xrightarrow{\sim} A^{\vee}$, called the product polarization. Recall that if $\alpha \in \End(E_1 \times E_2)$ we can describe $\alpha$ as a $2 \times 2$ matrix $(\alpha_{ij})$ by putting $\alpha_{ij}=p_i \alpha e_j \in \Hom(E_j,E_i)$, i.e.\ 
\[\End(E_1 \times E_2) = \Bigl\{ \begin{pmatrix}
    \alpha_{11} & \alpha_{12} \\ \alpha_{21} & \alpha_{22} 
\end{pmatrix} : \alpha_{ij} \in \Hom(E_j,E_i) \Bigr\}. \]
\begin{proposition} \textnormal{\cite[Prop.\ 60]{kani2016}} {\label{Kani's prop}}
    We have that \[\End_{\lambda_1 \times \lambda_2}(E_1 \times E_2) = \Bigl\{ \begin{pmatrix}
    \alpha_{11} & \alpha'_{21} \\ \alpha_{21} & \alpha_{22} 
\end{pmatrix} : \alpha_{ii} \in \End_{\lambda_i}(E_i), \alpha_{21} \in \Hom(E_1,E_2) \Bigr\}, \]
where $\alpha'_{21}=r_{\lambda_1, \lambda_2}(\alpha_{21})=\lambda_1^{-1} \alpha_{21}^{\vee}\lambda_2$. Thus, the map $(\alpha_1, \alpha_2, \beta) \mapsto$ $\begin{psmallmatrix}
    \alpha_1 & \beta' \\ \beta & \alpha_2 \end{psmallmatrix}$   defines an isomorphism: 
   \[ \phi : \End_{\lambda_1}(E_1) \oplus \End_{\lambda_2}(E_2) \oplus \Hom(E_1,E_2) \xrightarrow{\sim} \End_{\lambda_1 \times \lambda_2}(A), \] 
    equivalently written as \[ \phi : \NS(E_1) \oplus \NS(E_2) \oplus \Hom(E_1,E_2) \xrightarrow{\sim} \NS(A).\]
\end{proposition}
Finally, since $\NS(E) \cong \Z$ by the degree map, we can be more precise. 
\begin{corollary} \textnormal{\cite[Prop.~62]{kani2016}} \label{corollaryKani}  The above isomorphism  
    \[ \phi: \Z  \oplus \Z  \oplus \Hom(E_1,E_2) \xrightarrow{\sim} \NS(A)\]
     is explicitly given by \[(a, b, f) \mapsto (a-1) \mathrm{Im}(e_1) + (b- \deg f) \mathrm{Im}(e_2) + \Gamma_{f},\] where $\Gamma_f$ is the divisor associated to the graph $\{ (P,f(P))\}_{P \in E_1}$.
\end{corollary}

\begin{lemma}\label{inverseimagedivisor}
    Given an isogeny $f:E_2 \to E_1$, denote with $\Gamma^{-1}_f$ the divisor $\{(f(P),P) \}_{P \in E_2} \in \NS(E_1 \times E_2)$. Then the following relation holds 
    \[ \Gamma^{-1}_f= \Gamma_{f^{\vee}} + (\deg f -1) (\mathrm{Im}(\mathrm{e_1})-\mathrm{Im}(\mathrm{e_2})) \]
    as divisors in $\NS(E_1 \times E_2)$.
\end{lemma}
\begin{proof}
The divisor $\Gamma_{f}^{-1}$ corresponds to the divisor $\Gamma_f=\{P, f(P) \}_{P \in E_2} \in \NS(E_2 \times E_1)$. By Corollary \ref{corollaryKani}, it corresponds to $(1, \deg f, f) \in \NS(E_2) \oplus \NS(E_1) \oplus \Hom(E_2,E_1)$. Now, implicitly identifying $E_1$ and $E_2$ with their duals via $\lambda_1$ and $\lambda_2$, we have
\[ \phi(1, \deg f, f)= \begin{pmatrix}
   1 & f^{\vee} \\ f & \deg f 
\end{pmatrix} \in \End_{\lambda_2 \times \lambda_1}(E_2 \times E_1),\]
which corresponds to the element $ \begin{pmatrix}
   \deg f & f \\ f^{\vee} & 1
\end{pmatrix} \in \End_{\lambda_1 \times \lambda_2}(E_1 \times E_2)$. Corollary \ref{corollaryKani} implies the statement.
\end{proof}

\section{Transcendental Brauer group of \texorpdfstring{$Y_C$}{Y C}}

As mentioned in the introduction, this paper applies a  strategy similar to the one followed by Skorobogatov and Zarhin in \cite{skorobogatov2010brauer} in order to compute the transcendental Brauer group of the family of K3 surfaces studied by Van Luijk in \cite{vanluijk2007cubic} and which we now define again.

\subsection{Construction of K3 surfaces}
We will follow very closely \cite[\S 2]{vanluijk2007cubic}, trying to adopt the same notation as well.    

Let $k$ be a number field and let $C \subset \Proj^2_k$ be a smooth projective cubic curve. Let $P,Q,R$ be three (geometric) points of $C$; we say $P,Q,R$ are collinear if the divisor $(P)+(Q)+(R)$ is linear equivalent to a linear section of $C$. This simply extends the usual collinearity property in order to count tangency points and inflexion points with the suitable multiplicity. By Bezout's theorem, for every two points $P,Q$ of $C$, there exists a unique third point $R$ collinear with $P$ and $Q$ (possibly equal to $P$ or $Q$). This fact yields the following natural isomorphism 
\[ S:= C \times C \cong \{ (P,Q,R) \in C^3 : P,Q,R \text{ are collinear} \}.  \]

\noindent Let $\rho$ be the automorphism of $C^3$ that permutes cyclically the three coordinates; at the level of $C \times C $, $\rho$ maps $(P,Q)$ to $(Q,R)$ with $P,Q,R$ collinear. This defines an automorphism of $S$ of order 3. Let $X_C$ be the quotient $S/\rho$ and let $X=X_C$ when the curve $C$ is understood.

The argument on page 3 in \cite{vanluijk2007cubic} shows that $X_C$ has nine singular points, all double points, which correspond to the locus $S_3 \subset S$ of points $(P,P)$ where $P$ is one of the nine inflexion points of $C$. Let $Y_C$ (or $Y$ when the dependency on $C$ is clear) denote the blow-up of $X_C$ at its singular points; as the singular locus of $X_C$ is defined over $k$, so is $Y_C$. Moreover, defining $Y_0:=(S \setminus S_3) / \rho$, we have that $Y \setminus Y_0$ splits over $\Bar{k}$ into nine pairs of smooth (-2)-curves. The two curves within a pair intersect in one point and each pair corresponds to a geometric point of $S_3$. 

\begin{proposition}\textnormal{ \cite[2.2, 3.12]{vanluijk2007cubic}}
    The surface $Y_C$ is a smooth $K3$ surface with $\mathrm{rk}\NS(Y_C)=18+ \rk \End \Jac C$.
\end{proposition}
\subsection{Equivariant resolutions of singularities}
Let $S_0:=S \setminus S_3$. The goal of this section is to find a surface $S'$, birationally equivalent to $S$ (in particular containing an open subset isomorphic to $S_0$), endowed with a proper surjective morphism $S' \to Y$ such that the diagram 
\begin{equation} 
    \xymatrix{
S' \ar[d] \ar[r] & S \ar[d]^{\pi}\\
Y \ar[r] & S/\rho} \label{diagram} 
\end{equation}
commutes. Indeed, we want to extend the action of $\rho$ to $S'$ and have $S'/\rho$ be a resolution of singularities of $S/\rho$, so that we have a birational morphism $S'/\rho \to Y$.

Let us first consider the situation at the geometric level, i.e.\ over $\Bar{k}$. We recall the considerations from \cite{vanluijk2007cubic}. Every singularity of $S/\rho$ comes from a fixed point of $\rho$ on $S$. These take the form $(P,P)$, where $P$ is any of the nine inflexion points of $C$. To work out the local behaviour around these points, we choose two copies $r,s$ of a uniformizer for $C$ at $P$. These are then local parameters for $S$ at $(P,P)$. The local description of $\rho$ is $(r,s) \mapsto (s,-r-s)$ because it corresponds to the description of the formal group law modulo squares; this describes the action of $\rho$ on the tangent space of $S$ at $(P,P)$. The subring of $\rho$-invariants in $k[r,s]$ is generated by $a=r^2+rs+s^2$, $b=-3rs(r+s)$, $c=r^3+3r^2s-s^3$, which satisfy the equation $a^3=b^2+bc+c^2$. This shows that $S/\rho$ has 9 (geometric) singularities of type $A_2$.     

Up to a linear change over $\Bar{k}$, we can describe the tangent space of $S$ at a fixed point of $\rho$ as $\mathbb{A}^2=\Spec \bar{k}[x,y]$ with $\rho$ acting by the matrix $\begin{pmatrix} \zeta_3 & 0 \\ 0 & \zeta_3^2\end{pmatrix}$, where $\zeta_3$ is a primitive cube root of unity. Let $S''$ be the blow-up of $S$ at $S_3$. The exceptional divisor above every point in $S_3$ is isomorphic to $\Proj^1_{\bar{k}}$ with the action of $\rho$ given as a projective transformation by the same matrix. Hence each exceptional divisor contains two fixed points. Consider a chart $U_0$ containing a fixed point $P_0$. Up to an étale map, $S''$ can be described by $\{xT_1=yT_0\}$ in $\mathbb{A}^2_{x,y} \times \Proj^1_{[T_0:T_1]}$. The open set $U_0$ is given by $\{ xT_1=y \}$ in $\mathbb{A}^3_{x,y,T_1}= S'' \cap \{T_0=1 \}$, where $P_0$ is $\{x=y=T_1=0\}$. Clearly $U_0 \cong \mathbb{A}^2_{x,T_1}$ and $\rho$ acts on $U_0$ by the matrix $\begin{pmatrix} \zeta_3 & 0 \\ 0 & \zeta_3\end{pmatrix}$. Therefore, the blow-up of $S''$ at $P_0$ has an exceptional divisor that is pointwise fixed by $\rho$. The reasoning is analogous for an open chart around the point $P_1=\{x=y=T_0=0 \}$. 

In summary, let $S''$ be the blow-up of $S$ at $S_3$ and let $S'$ be the blow-up of $S''$ at the locus of fixed points of the induced action of $\rho$. The computation above shows that the locus of $\rho$-fixed points in $S'$ is a divisor, so by the Chevalley--Shepard--Todd theorem the quotient $S'/\rho$ is smooth. Since $Y$ is the minimal resolution of $S/\rho$, we obtain a birational morphism $S'/\rho \to Y$, which fits into the diagram \ref{diagram}. Finally, the construction of $S'$ descends to $k$, since the blow-up locus is Galois stable.

\begin{remark}
The additional geometry of $S'$ does not affect the Brauer group, since it derives from blow-ups of the regular surface $S$. Similarly, the birational morphism $S'/\rho \to Y$ induces an isomorphism of the Brauer groups (see \cite[Cor.\ 7.2]{grothendieck}).
\end{remark}

\begin{lemma}{\label{brauer_absurface}}
The inclusion $S_0 \hookrightarrow S'$ induces an isomorphism 
$ \br(S') \xrightarrow{\sim} \br(S_0)$.
\end{lemma}
\begin{proof}
  Since the map $S' \to S$ is a birational morphism of smooth surfaces, the induced maps $ \br(S) \to \br(S')$ and $\br(\Bar{S}) \to \br(\Bar{S'})$ are isomorphisms (\cite[Cor.\ 7.2]{grothendieck}). The composition of these injective maps 
  \[ \br(S) \xrightarrow{\sim} \br(S') \xrightarrow{} \br(S_0)  \]
  is an isomorphism by Lemma \ref{general_lemma}, so the restriction map $\br(S') \to \br(S_0)$ is an isomorphism. 
    
\end{proof}

\begin{lemma}{\label{galoiscovering}}
    The map $\pi : S_0 \to Y_0:=S_0/\rho$ is a finite flat covering of regular surfaces such that $k(S_0)/k(Y_0)$ is Galois of degree 3. The same holds for the map $\pi: S' \to S'/\rho$. 
\end{lemma}

\begin{proof}
 Every finite surjective map of regular varieties is flat, so $\pi : S_0 \to Y_0$ is flat (from \cite[Remark\ 4.3.11]{liu2006algebraic}).  
Moreover, $\langle \rho \rangle= \Aut_{Y_0}(S_0)$ induces nontrivial automorphism of $k(S_0)$ over $k(Y_0)$, because we can find a coordinate $f \in k(S_0)$ such that $f(x)=0$ and $f(\rho(x)) \neq 0$ for some $x \in S_0$.  Now $k(S_0) / k(S_0)^{\langle \rho \rangle}$ is a Galois extension of degree $3$ and, since $k(Y_0) \subset k(S_0)^{\langle \rho \rangle} \subset k(S_0)$ with $k(S_0)/k(Y_0)$ degree 3, we have that $k(Y_0) = k(S_0)^{\langle \rho \rangle}$.
Since the function field is a birational invariant, the same conclusions hold for the map of regular surfaces $S' \to S'/\rho$.

\end{proof}

\noindent The surface $S'$ allows us to relate the Brauer groups of $Y_0=S_0/\rho$ and $Y$.
\begin{lemma}{\label{lemma_maledetto}} For every $n$, we have a natural injection of groups \[ \br(Y)_n \hookrightarrow \br(Y_0)_n, \] which is an isomorphism for $n$ coprime to 3. 
\end{lemma}
\begin{proof}
The inclusion $Y_0 \subset Y$ induces an injection  $\br(Y)\hookrightarrow \br(Y_0)$. Moreover, by Proposition \ref{galois_covering} and Lemma \ref{galoiscovering} we have $\br(Y)_n \cong \br(S')^{\rho}_n$ for $(n,3)=1$. We also have $\br(S')^{\rho}_n \cong \br(S_0)^{\rho}_n$ by Lemma \ref{brauer_absurface} and, finally, $\br(S_0)^{\rho}_n \cong \br(Y_0)_n$ by Proposition \ref{galois_covering} again. 
\end{proof}

\subsection{Geometric Brauer groups}

In order to relate the transcendental Brauer groups of $S$ and $Y$, we first need to study the relationship between their geometric Brauer groups.

 Let $A$ be the abelian variety $E \times E$, where $E= \Jac(C)$, so that $\Bar{S} \cong \Bar{A}$. We start by explicitly computing $\h^1(\rho, \NS(\Bar{A}))$ using the following description of $\NS(\bar{E} \times \bar{E})$: 
\[ \NS(\bar{E} \times \bar{E}) \cong \Z \cdot  e \oplus \End(\bar{E}) \oplus \Z \cdot e' \]
where $e$ and $e'$ are the divisors $\bar{E} \times \{ 0_E\}$ and $\{ 0_E \} \times \bar{E}$, respectively. 

\begin{lemma} \label{computationNS}
We have $\h^1(\rho, \NS(\Bar{A}))=0$.
\end{lemma}
\begin{proof}
 
 The group $\h^1(\rho, \NS(\Bar{A}))$ can be described as \[\{ x \in \NS(\Bar{A}) : (1+\rho+\rho^2)(x)=0 \}/(\rho -1)\NS(\Bar{A}).\]
Our computational strategy will be to show that $(\rho -1)\NS(\Bar{A})$ is a primitive sublattice of $\NS(\Bar{A})$ and that  $(1+\rho+\rho^2)(x) \neq 0$ for all $x \in \NS(\Bar{A}) \setminus (\rho -1)\NS(\Bar{A})$. 
\\
Let $\phi: \Z \cdot e  \oplus \Z \cdot e'   \oplus \End(\bar{E}) \to \NS(\bar{A})$ be the isomorphism given by Corollary \ref{corollaryKani} \[(a,  b, f) \mapsto (a-1) e + (b- \deg f) e' + \Gamma_{f}\] where $\Gamma_f$ is the divisor associated to the graph $\{ (P,f(P))\}_{P \in E}$.   We study the induced action of $\rho$ on the left-hand side using $\phi$ and the action of $\rho^*$ on $\NS(\bar{A})$. 

First, suppose that $E$ has CM by an order $\oo$, which is generated by $\alpha$. Let $x^2+cx+d \in \Z[x]$ be the minimal polynomial of $\alpha$. Note that $\Gamma_0=e$.
\begin{itemize}
    \item $\rho(1,0,0)=\rho(\phi^{-1}(\Gamma_0)) = \phi^{-1}(\rho^*(\Gamma_0))=\phi^{-1} (\Gamma_{-1})=(1, 1, -1) $;
    \item $\rho(0,1,0)= \phi^{-1}(\rho^*(e'))=\phi^{-1}(e)=(1,0,0)$;
    \item $\rho(0,0, 1)=\phi^{-1}(\rho^*(-e + \Gamma_{1}-e' )) = (-1,-1, 1) + (-1,0,0) + \phi^{-1}(\rho^*(\Gamma_{1}))$; we have that $\rho^*(\Gamma_{1})$ is the divisor supported on $\Gamma^{-1}_{-2}=\{(-2P,P)\}_{P \in E}$. Using Lemma \ref{inverseimagedivisor}, we have $\phi^{-1}(\rho^*(\Gamma_{1}))=(4,1,-2)$, thus 
    $\rho(0,0,1)= (-2, -1, 1) + (4,1,-2)= (2,0,-1)$;
    \item $\rho(0,0, \alpha)=\phi^{-1}(\rho^*(-e+\Gamma_{\alpha}-d\cdot e'))= (-1,-1,1) + (-d,0,0) +  \phi^{-1}(\rho^*(\Gamma_{\alpha}))$; as before, Lemma \ref{inverseimagedivisor} implies $\rho^*(\Gamma_{\alpha}) = \Gamma^{-1}_{-1-\alpha}=\phi^{-1}(1-c+d, 1, (-1-\alpha)^{\vee})$. Using that $(-1-\alpha)^{\vee}=-1+ \alpha +c$, we obtain $\rho(0,0,\alpha)=(-d-1,-1,1)+(1-c+d,1,-1+\alpha+c)=(-c,0,\alpha+c)$.
\end{itemize}
The sublattice $(\rho -1)\NS(\Bar{A})$ is the primitive sublattice generated by $(0, 1, -1)$ and $(1,0,-1)$; a direct complement of it in $\NS(\Bar{A})$ is the sublattice generated by $(0,1,0)$ and $(0,0,\alpha)$. Calculations lead to $(1+\rho+\rho^2)(0,1,0)=(2,2,-1)$ and $(1+\rho+\rho^2)(0,0,\alpha)=(-c,-c, 3\alpha+2c)$, hence any nontrivial $\Z$-linear combination of these cannot vanish. 

\noindent The computation is analogous for a non-CM elliptic curve $E$. 
\end{proof}

\begin{proposition}{\label{geometric_brauer}}
    The natural map $\pi^*: \br(\Bar{Y}_0) \to \br(\Bar{S}_0)$ is an isomorphism, which induces an isomorphism of $\Gamma$-modules $\br(\Bar{Y}) \to \br(\Bar{S})$; in particular, the action of $\rho$ on $\br(\Bar{S})$ is trivial.
\end{proposition}
\begin{proof}

 Consider the short exact sequence 
\[ 0 \to A^{\vee}(\Bar{k}) \to \Pic(\Bar{A}) \to \NS(\Bar{A}) \to 0 \]
on which $\rho$ acts via pullbacks. Thus, we have a long exact sequence of group cohomology involving the terms
\[ \cdots \to \h^1(\rho, A^{\vee}(\Bar{k})) \to \h^1(\rho, \Pic(\Bar{A})) \to \h^1(\rho, \NS(\Bar{A})) \to \cdots. \]
Notice that $\rho$ acts on $E \times E$ by mapping $(P,Q) \mapsto (Q, -P-Q)$, so on $(E \times E)^{\vee}$ by mapping $(P,Q) \mapsto (-P-Q,P)$. Then, $(\rho -1)(P,Q)=(-2P-Q, P-Q)$, which is surjective on $A^{\vee}(\Bar{k})$, hence $\h^1(\rho, A^{\vee}(\Bar{k}))=0$. 
Using Lemma \ref{computationNS}, the exact sequence in group cohomology shows that $\h^1(\rho, \Pic(\Bar{A}))=0$.

By Lemma \ref{galoiscovering} $\pi: S_0 \to Y_0$ is a flat Galois covering with Galois group $\Z /3$, therefore there is the associated Hochschild--Serre spectral sequence (\cite[Thm 14.9]{milneLEC})
\[ \h^p(\Z /3, \h^q_{\et} (\Bar{S}_0, \mathbb{G}_m)) \Longrightarrow \h^{p+q}_{\et}(\Bar{Y}_0, \G_m),\] 
which yields the following exact sequence: 
\[ \cdots \to \h^2(\rho, \h^0_{\et}(\Bar{S}_0,\G_m)) \to \ker\bigg( 
\h^2_{\et}(\Bar{Y}_0, \G_m) \to \h^0(\rho, \h^2_{\et}(\Bar{S}_0,\G_m)) \bigg) \to \h^1(\rho, \h^1_{\et}(\Bar{S}_0,\G_m)) \to \cdots. \]
Moreover, \[\h^2(\rho, \h^0_{\et}(\Bar{S}_0,\G_m))= \h^2(\rho, \Bar{k}^{\times})= \Bar{k}^\times/\Bar{k}^{\times 3} =0,\] \[\h^1(\rho, \h^1_{\et}(\Bar{S}_0,\G_m))=\h^1(\rho, \Pic(\Bar{S}))=0,\] 
so we deduce an injection $\br(\Bar{Y}) \hookrightarrow \br(\Bar{Y}_0) \hookrightarrow \h^0(\rho, \br(\Bar{S}_0)) \hookrightarrow \br(\Bar{S}_0) \cong \br(\Bar{S})$.
To conclude the proof, it is sufficient to notice that $\br(\Bar{Y})$ and $\br(\Bar{S}))$ are both abstractly isomorphic to $(\Q/\Z)^2$ since $b_2(\Bar{Y})=22$, $b_2(\Bar{S})=6$ from general facts on K3 surfaces/abelian surfaces and $\rho(\Bar{Y})=16+\rho(\Bar{S})$ by \cite[Prop.\ 3.11, 3.12]{vanluijk2007cubic}.
\end{proof}
\begin{remark}
Consider the exact sequence coming from the Kummer sequence \[ 0 \to \NS(\Bar{S}) \otimes \Z/l^n \to \h^2_{\et}(\Bar{S},\mu_{l^n}) \to \br(\Bar{S})_{l^n} \to 0. \]
In the Kummer surface case, the action of the involution $[-1]_{\Bar{S}}$ induces a trivial action on $\h^2_{\et}(\Bar{S},\mu_{l^n}) = \bigwedge^{2} \h^1_{\et}(\Bar{S},\mu_{l^n})$, whereas the action of $\rho$ on $\h^2_{\et}(\Bar{S},\mu_{l^n}) $ is not trivial because it is nontrivial on $\NS(\Bar{S})$. Nevertheless, we are lucky since the nontriviality of the action of $\rho$ comes only from the Néron--Severi group, i.e.\ the action on the quotient $\br(\Bar{S})_{l^n}$ is still trivial and this allows a result analogous to \cite[Thm 2.4]{skorobogatov2010brauer}.     
\end{remark}

\begin{remark} \label{remark_SVK}
The proof of Proposition \ref{geometric_brauer} also shows that the injection $\br(\Bar{Y}) \hookrightarrow \br(\Bar{Y_0})$ is an isomorphism. This can also be shown directly using the sequence 
\[ 0 \to \br(\Bar{Y}) \to \br(\Bar{Y_0}) \to \bigoplus \h^1_{\et} (Z_i, \Q / \Z),  \]
where the sum ranges over all connected components of $\bar{Y}\setminus \bar{Y}_0$. Indeed, every $Z_i$ is isomorphic to $\Proj^1_{\bar{k}} \vee \Proj^1_{\bar{k}}$ (because it arises from the resolution of a $A_2$-singularity), hence $\h^1_{\et} (\Proj^1_{\bar{k}} \vee \Proj^1_{\bar{k}}, \Q / \Z)=0$ since $\pi_1^{top}(\Proj^1_{\C} \vee \Proj^1_{\C}, \Bar{x})=0$ by the Seifert--Van Kampen Theorem.
 \end{remark}

\subsection{Transcendental Brauer groups}

A final lemma before proving the main result of this section:

\begin{lemma}\label{mainlemma}
    For every $n$, we have a natural injection \[ \frac{\br(Y)_n}{\br_1(Y)_n} \hookrightarrow \frac{\br(Y_0)_n}{\br_1(Y_0)_n},  \] which is an isomorphism for $n$ coprime to 3.
\end{lemma}
\begin{proof}
 For every $n$, we have 
 \[ \xymatrix{
\br(Y)_n \ar[d] \ar[r]^{i^*} & \br(Y_0)_n \ar[d]\\
\br(\Bar{Y})_n \ar[r]^{\sim} & \br(\Bar{Y_0})_n}
\] 
where $i^*: \br(Y)_n \to \br(Y_0)_n$ is the injective map induced by the inclusion. The isomorphism in the diagram follows from Remark \ref{remark_SVK}. In order to quotient by the algebraic Brauer groups, we need  \[i^* \br(Y)_n \cap \br_1(Y_0)_n= i^* \br_1(Y)_n.  \]
From one side we have $i^* \br_1(Y)_n \subset i^* \br(Y)_n$ and $i^* \br_1(Y)_n \subset \br_1(Y_0)_n$. On the other hand, if $i^*x \in i^* \br(Y)_n \cap \br_1(Y_0)_n $ then, $i^*\Bar{x}=0$ hence $\Bar{x}=0$ by the isomorphism, implying $x \in \br_1(Y)_n$. 
The statement for $n$ coprime to 3 follows by Lemma \ref{lemma_maledetto}.
\end{proof}

\begin{theorem} \label{main_thm}
    The map $\pi^*: \br(Y_0) \to \br(S_0)$ induces an embedding 
    \[ \frac{\br(Y)_n}{\br_1(Y)_n} \hookrightarrow \frac{\br(S)_n}{\br_1(S)_n},  \]
    which is an isomorphism if $n$ is coprime to 3. The subgroups of elements of order coprime to 3 of the transcendental Brauer groups $\br(Y)/\br_1(Y)$ and $\br(S)/\br_1(S)$ are isomorphic.  
\end{theorem}
\begin{proof}

By Proposition \ref{geometric_brauer}, the  commutative diagram
\[ \xymatrix{
\br(Y_0)_n \ar[d] \ar[r] & \br(S_0)_n \ar[d]\\
\br(\Bar{Y_0})_n \ar[r]^{\sim} & \br(\Bar{S_0})_n}
\]
 gives the embedding $\br(Y_0)_n/\br_1(Y_0)_n \hookrightarrow \br(S_0)_n/\br_1(S_0)_n \cong \br(S)_n/\br_1(S)_n$. Hence Lemma \ref{mainlemma} implies the existence of the required embedding for all $n$.  
 
Consider now $(n,3)=1$. The natural map $\br(S) \to \br(\Bar{S})$ naturally commutes with $\rho$. Since $\rho$ acts trivially on $\br(\Bar{S})$ by Proposition \ref{geometric_brauer},  we have that $\rho(\Bar{x})-\Bar{x}=0$ for any $x \in \br(S)_n $. Therefore $\rho(x)-x \in \br_1(S)_n$ and, similarly, $\rho^2(x)-x \in \br_1(S)_n$. 
Then, \[ 3x = (x+ \rho(x) +\rho^2(x)) + (x - \rho(x))+(x-\rho^2(x)) \in \br(S)_n^{\rho} + \br_1(S)_n\]
and the assumption that $(n,3)=1$ implies that $\br(S)_n=\br(S)_n^{\rho} + \br_1(S)_n$, hence 
\[ \br(S)_n/\br_1(S)_n = \br(S)_n^{\rho}/\br_1(S)_n^{\rho}. \]
Moreover, we obtain the  diagram   
\[ \xymatrix{
\br(Y_0)_n \ar[d] \ar[r]^{\sim} & \br(S)_n^{\rho} \ar[d]\\
\br(\Bar{Y_0})_n \ar[r]^{\sim} & \br(\Bar{S})_n}
\]
by applying Proposition \ref{galois_covering} to the triple covering $\pi:S_0 \to Y_0$ together with $\br(S_0)=\br(S)$ by Lemma \ref{brauer_absurface}.
Combining the diagram with Lemma \ref{mainlemma} proves the first statement. 

Finally, the second statement follows from the claim that an element of order coprime to 3 in $\br(Y)/\br_1(Y)$ comes from $\br(Y)_n$ for some $n$ coprime to 3. The proof of this claim only uses that the Brauer group is torsion. \\
Indeed, let us fix $x \in [x] \in (\br(Y)/\br_1(Y))_n$ with $(n,3)=1$; we have $n \cdot x \in \br_1(Y)$, so there exists $m>0$ such that $mn\cdot x=0$, so $x \in \br(Y)_{mn}$. Let $m=3^am'$ with $(m',3)=1$; there exists $a'>a$ such that 
\[ x':= 3^{a'} \cdot x \in [x] \in (\br(Y)/\br_1(Y))_n\]
 by considering any $a'$ such that $3^{a'} \equiv 1 \pmod{n}$. The proof is concluded, since $x' \in [x]$ and $ x' \in \br(Y)_{m'n}$ with $(m'n,3)=1$.

\end{proof}

\section{Transcendental Brauer group of \texorpdfstring{$S$}{S}}

In the previous section we have related the transcendental Brauer groups  of $S=C \times C$ and of the K3 surface $Y$. Let us now study the former, exploiting the fact that $S$ is a torsor of an abelian surface. 

In this section, we assume $C$ to be a diagonal cubic curve: $ax^3+by^3+cz^3=0$, defined over $\Q$, so we can assume $a,b,c$ to be a primitive triple of integers.  Let $L$ be a cubic extension of $\Q$ such that $C_L$ has a rational point, e.g.\ the number field $\Q(\sqrt[3]{\frac{a}{b}})$. We have $C_L \cong E_L$, where $E=\Jac(C)$ is described by $y^2=x^3-2^43^3a^2b^2c^2$. The equation of the Jacobian derives from classical results like \cite{AN2001304}.  
Notice that this elliptic curve has CM by $\oo_K$ with $K=\Q(\sqrt{-3})$, as expected from the diagonal form of the curve $C$. 

The compatibility of the restriction maps 
\[ \xymatrix{ \br(C \times C) \ar[r] \ar[dr] & \br(E_L \times E_L) \ar[d] \\ & \br(\bar{C} \times \bar{C})} \]
 implies the injection $\br(C \times C)/\br_1(C \times C) \hookrightarrow \br(E_L \times E_L)/\br_1(E_L \times E_L)$.
In full generality, we would need to study the right hand side, but, away from 3, we can reduce this to $\br(E \times E)$ thanks to the following proposition.  

\begin{proposition}\textnormal{\cite[3.3]{uniformK3} }  \label{VAVisom}
Let $N$ be a positive integer, let $X$ and $X'$ be principal homogeneous spaces
of an abelian variety $A$, defined over the same field, and let $f : X' \to X$ be an $N$-covering. Then, for
any integer $n$ coprime to $N$, we have the isomorphism
\[f^*: \frac{\br(X)_n}{\br_1(X)_n} \xrightarrow{\sim} \frac{\br(X')_n}{\br_1(X')_n}. \]
\end{proposition}

Fortunately, the transcendental Brauer groups of products of CM elliptic curves have been extensively studied in $\cite{rachel2015}, \cite{corrigendumRachel}$. We now record the required results from that paper, stated directly under our assumptions on $L$ and $K$.  

\begin{theorem}\textnormal{\cite[2.9 \& 2.2]{rachel2015}}{\label{main thm Rachel}} Let $E/L$ be an elliptic curve with CM by $\oo_K$; let $l$ be a prime number, then \[ \bigg(\frac{\br(E \times E )}{\br_1(E \times E )} \bigg)_{l^{\infty}}= \bigg(\frac{\br(E \times E )_{l^m}}{\br_1(E \times E )_{l^m}} \bigg) \cong \Z/l^m, \] 
where $m=m(l)$ is the largest integer $k$ such that every prime $\mathfrak{q}$ of $KL$ of good reduction for $E_{KL}$ and coprime to $l$ satisfies $\psi_{E/KL}(\mathfrak{q}) \in \oo_{l^k}=\Z + l^k \oo_K$ for the Hecke character $\psi_{E/KL}$. Let $n(l)$ be the largest integer $k$ such that the ring class field $K_{l^k}$ of the order $\oo_{l^k}$ embeds into $KL$; then, $m(l) \leq n(l)$.  
\end{theorem}

Moreover, Newton explicitly computes the transcendental Brauer group of $E \times E$ for a CM elliptic curve $E: y^2=x^3+D$ defined over $\Q$. 

\begin{proposition}\textnormal{\cite[Example\ 3]{rachel2015}} \label{rachelexample} 
Let $E$ be the elliptic curve over $\Q$ with equation $y^2=x^3+D$ with $D \in \Z \setminus \{ 0\}$. Then 
\[ \frac{\br(E \times E)}{\br_1(E \times E)} = \begin{cases} \Z/2\Z &\text{ if } D \text{  is a   cube in } \Z; \\
  \Z/3\Z &\text{ if } 4D \text{  is a   cube in } \Z;  \\ 0 &\text{ otherwise.}
\end{cases} \]
\end{proposition}
We recall the following property of Hecke characters:
\begin{lemma}{\label{grossencharacter}}
    Let $E/K$ be an elliptic curve with CM by $\oo_K$ and let $L$ be a number field. Let $\mathfrak{q}$ be a prime in $KL$ lying over the prime $\mathfrak{p}$ in $K$. Then, \[\psi_{E/KL}(\mathfrak{q})=\psi_{E/K}(\mathfrak{p})^{f_{\mathfrak{q}/\mathfrak{p}}},\] where $f_{\mathfrak{q}/\mathfrak{p}}$ is the relative inertia degree.   
\end{lemma}
\begin{proof}
    This derives from the definition of the Hecke character $\psi_{E}$ and the compatibility of the Artin map with the norm map, see \cite[II.9.1, II.3.5]{silverman1994advanced}. 
\end{proof}
Finally, we need an elementary lemma before computing the transcendental Brauer group of $C \times C$ in our case of interest. 
\begin{lemma}{\label{juggling cubes}}
    Let $a,b,c$ be positive integers such that $abc$ is not a cube in $\Z$. Then, there exists a choice of $\lambda \in \{a/b,b/c, c/a\}$ such that $\Q(\sqrt[3]{abc}) \neq \Q(\sqrt[3]{\lambda})$.   \end{lemma}
\begin{proof}
First recall that  $\Q(\sqrt[3]{abc}) \neq \Q(\sqrt[3]{\lambda})$ if and only if $abc \not\sim \lambda$ and $abc \not\sim \lambda^2$ in $\Q^{\times}/\Q^{\times 3}$.
In particular, if $\lambda=a/b$, the conditions become $b^2c \not\sim 1$ and $a^2c \not\sim 1$. If $\lambda=b/c$, then the conditions become $ac^2\not\sim 1$ and $ab^2 \not\sim 1$. Finally, if $\lambda=c/a$, they are $a^2b \not\sim 1$ and $bc^2 \not\sim 1$. 

Notice that at least one between $b^2c \not\sim 1$ and $a^2c \not\sim 1$ is satisfied, since $abc \not \sim 1$. If they are both satisfied, choose $\lambda=a/b$. Otherwise let us start by assuming $b^2c \sim 1$ and therefore $a^2c \not\sim 1$; using again that $abc \not\sim 1$, we have $ab^2 \not\sim 1$, so $\lambda=b/c$ works. On the other hand, if we assume $a^2c \sim 1$ and therefore $b^2c \not\sim 1$, then $a^2b \not \sim 1$ and $\lambda= c/a$ works. 
\end{proof}

\begin{theorem}{\label{thm on C x C}}
    Let $C/\Q$ be a diagonal cubic curve $ax^3+by^3+cz^3=0$  such that $abc$ is not a cube in $\Q$. Then, \[\frac{\br(C \times C)}{\br_1(C \times C)}= \begin{cases}
    \Z/2\Z &\text{ if } 4abc \text{ is a cube in } \Q; \\
   0 &\text{ otherwise.}
\end{cases} \]
\end{theorem}
\begin{proof}
    If $C(\Q) \neq \emptyset$, then $C \cong E= \Jac(C)$ and therefore Proposition \ref{rachelexample} applies directly with $D=-2^43^3a^2b^2c^2$. We can therefore assume that $C(\Q) = \emptyset$; the proof is still a modification of the reasoning in \cite[Example\ 3]{rachel2015}.
    
    The statement for the $l$-primary part of the transcendental Brauer group for $l \neq 3$ is a direct implication of Proposition \ref{VAVisom} and Proposition \ref{rachelexample}.  
    Again let $D=-2^43^3a^2b^2c^2$ and let $\mathfrak{q}$ be a prime in $KL$ coprime to $D$; denote by $\pi_{\mathfrak{p}}$ the unique generator of $\mathfrak{p}=\mathfrak{q} \cap \oo_K$ such that $\pi_{\mathfrak{p}} \equiv 1 \pmod{3}$. Then, Lemma 
    \ref{grossencharacter} and \cite[Example II.10.6]{silverman1994advanced} show that \[ \psi_{E/KL}(\mathfrak{q})= \bigg[  \legendre{4D}{\pi_{\pp}}^{-1}_6 \cdot \pi_{\pp}  \bigg]^{f_{\mathfrak{q}/\mathfrak{p}}} , \]
    where $\legendre{\cdot}{\cdot}_6$ denotes the sextic residue symbol on $\Z[\zeta_3]$. 

    Let $l=3$.  We have assumed that $4D$ is not a cube in $\Z$. Moreover, choose $L=\Q(\sqrt[3]{\lambda})$ given by Lemma \ref{juggling cubes}, so that $KL \neq K(\sqrt[3]{4D})$ (indeed fixed an embedding of $KL$ into $\C$, $L$ is the only real subfield of degree 3 over $\Q$). Since $4D$ is not a cube in $K$, then by \cite[Ex 6.1]{cassels2010algebraic} there are infinitely many primes $\mathfrak{p}$ of $K$ such that $4D$ is not a cube mod $\mathfrak{p}$, which is equivalent to $\mathfrak{p}$ being inert in $K(\sqrt[3]{4D})/K$. Again by \cite[Ex 6.1]{cassels2010algebraic}, we can also assume that $\mathfrak{p}$ is split in $KL/K$, since $KL \neq K(\sqrt[3]{4D})$. Given all these choices, we have found a prime $\qq$ in $KL$ above $\pp$  such that $f_{\qq/\pp}=1$ and $\legendre{4D}{\pi_{\pp}}_6 \neq \pm 1$, therefore  $\psi_{E/KL}(\mathfrak{q}) \not\in \oo_3=\Z+3\oo_K$, giving $m(3)=0$ in Theorem \ref{main thm Rachel}.   
\end{proof}

\begin{corollary}\label{main corollary}
    Let $C/\Q$ be a diagonal cubic curve $ax^3+by^3+cz^3=0$  such that $abc$ is not a cube in $\Q$. Then, 
   \[\frac{\br(Y_C)}{\br_0(Y_C)}= \begin{cases}
    \Z/2\Z &\text{ if } 4abc \text{ is a cube in } \Q; \\
   0 &\text{ otherwise.}
\end{cases} \]
\end{corollary}
\begin{proof}
    By comparing Theorems \ref{main_thm}  and \ref{thm on C x C}, we obtain $\br(Y_C) / \br_1(Y_C)$ and by \cite[Prop.\ 4.2]{vanluijk2007cubic}, we have that $\br(\Q) \to \br_1(Y_C)$ is surjective.
\end{proof}

\begin{remark}
This result agrees with  V\'arilly-Alvarado's conjecture on the uniform boundness of the Brauer groups of K3 surfaces, see \cite[5.5]{anthonyconj}. Moreover, the cardinality of the transcendental Brauer group of $Y_C$ can be bounded depending only on the base field $k$, thanks to the bound of $n(l)$, given by class field theory as discussed in \cite[Remark\ 1]{rachel2015}. Indeed, V\'arilly-Alvarado's conjecture has been proven for CM K3 surfaces (such as our case) in \cite{Orr_Skorobogatov_2018}.
\end{remark}

\section{On the Hasse principle}

We can now discuss the implications of the computation of the Brauer group of $Y_C$ assuming Skorobogatov's conjecture for these K3 surfaces.

\begin{conj}[Skorobogatov, \cite{skorobogatov_2009}] \label{skorobogatovconj}
The Brauer--Manin obstruction is the only obstruction to the Hasse principle for algebraic K3 surfaces over number fields.    
\end{conj}

\begin{theorem}\label{galoiscubicpoints}
    Let $C/\Q$ be an everywhere locally soluble smooth curve with equation $ax^3+by^3+cz^3=0$. Then, Skorobogatov's conjecture implies the existence of Galois cubic points on $C$. 
\end{theorem}

\begin{proof}
Section 2 in $\cite{vanluijk2007cubic}$ describes how to relate $k$-rational points on $Y_C$ and $C$, for a field $k$ of characteristic zero. First, the surface $Y_C$ has a $k$-rational point if and only if $S/\rho$ does. 
Secondly, a $\bar{k}$-point of $S/\rho$, described by the orbit
 $\{(P,Q), (Q,R), (R,P) \}$, is $k$-rational if either  
\begin{itemize}
    \item $P,Q,R$ are $k$-rational collinear points on $C$, or 
    \item $P,Q,R$ are collinear $l$-rational points on $C$, where $l/k$ is a Galois cubic extension and $P,Q,R$ are permuted by $\Gal(l/k)$.
\end{itemize}     
If neither $abc$ nor $4abc$ is a cube in $\Q$, then Corollary \ref{main corollary} and Conjecture \ref{skorobogatovconj} imply the Hasse principle for $Y_C$, which implies the existence of Galois cubic points on $C$. If $abc$ or $4abc$ is a cube in $\Q$, a simple infinite descent argument implies that the everywhere locally soluble $C$ must have a rational point, see \cite[Remark 4.3]{vanluijk2007cubic} for details.
\end{proof}

Let us finally consider the case where $Y_C$ is everywhere locally soluble, which, as explained in the proof of Theorem \ref{galoiscubicpoints}, is a weaker assumption than $C$ being everywhere locally soluble. We are able to obtain the same result as in Theorem \ref{galoiscubicpoints}, assuming that $abc$ is not a rational cube. 

\begin{theorem} \label{NoObstructionThm}
Let $C/\Q$ be a smooth curve with equation $ax^3+by^3+cz^3=0$, with $abc$ not a rational cube, such that the surface $Y_C$ is everywhere locally soluble. Then, there is no Brauer--Manin obstruction to the Hasse principle on $Y_C$. Consequently, Skorobogatov's conjecture implies the existence of Galois cubic points on $C$.   
\end{theorem}
\begin{proof}
If $4abc$ is not a rational cube, then the result follows from Corollary \ref{main corollary}. Assume now that $4abc$ is a rational cube. Then Corollary \ref{main corollary} shows that there is only one nontrivial element $\beta$ of $\br(Y_C)/\br_0(Y_C)$ whose evaluation maps need to be studied. By \cite[Prop. 2.12]{rachel2015}, the element $\beta$ corresponds to the nontrivial element in $\br(\bar{E} \times \bar{E})_{2^{\infty}}^{\Gamma_{\Q}} \cong \Z/2\Z$. The discussion in \cite[\S 3]{skorobogatov2010brauer} describes explicit generators of $\br(\bar{E} \times \bar{E})_2$ as Azumaya algebras. In our case, we claim that $\beta$ corresponds to the Azumaya algebra $(x-3, u-3) \in \br(\bar{E} \times \bar{E})_{2}^{\Gamma_{\Q}} $, where the two components of $E \times E$ are given by $y^2=x^3-27$ and $v^2=u^3-27$. The nontriviality of this element will be a consequence of the surjectivity of the evaluation map at the prime 2. 

First notice that the evaluation of $\beta$ at the origin of $E \times E$ is trivial. Now consider the point $P=(3,0)$, which is the only nontrivial rational 2-torsion point of $E$. The evaluation of $\beta=(x-3, u-3)=(x^2+3x+9, u^2+3u+9)$ at $(P,P)$ is the quaternion algebra $(3,3)$ which ramifies at the primes $p=2,3$. We have then shown that $ev_{\beta}: (E \times E) (\Q_2) \to \Z/2$ is surjective; the strategy now is to chase the points via the natural maps to prove that the surjectivity holds for $Y=Y_C$ as well. We will adopt a slight abuse of notation by denoting with $\beta$ the  Brauer group element of interest regardless of it living inside $\br(C \times C), \br(E \times E), \br(S')$, or $\br(Y)$. This is justified by the discussion of these Brauer groups in Sections 3 and 4.   

Note that $C(\Q_2) \neq \emptyset$; this is a simple consequence of the fact that every $t \in \Z_2^{\times}$ is a cube in $\Z_2$ and the assumption that $4abc$ is a rational cube. Hence, there exists a $\Q_2$-isomorphism of $C_{\Q_2}$ and $E_{\Q_2}$ which makes the 3-covering map $C \to E$ into the multiplication $[3]: C_{\Q_2} \to E_{\Q_2}$; the same holds for $C \times C$ and $E \times E$. Since $\beta$ has order 2, it is preserved by the induced multiplication map on the Brauer groups, and, similarly, the origin and $(P,P)$ as $\Q_2$-rational points of $(E \times E)_{\Q_2}$ are also preserved by the map $[3]$. Therefore, the evaluation map $ev_{\beta}: (C \times C)(\Q_2) \to \Z/2$ is surjective. The map on Brauer groups induced by $S' \to C \times C$ preserves $\beta$ and therefore $ev_{\beta}: S'(\Q_2) \to \Z/2$ is also surjective. Indeed, in the case where the locus $S_3$ of the blow-up has a $\Q_2$-rational point, the exceptional divisor above that point is a pair of genus 0 curves intersecting in a $\Q_2$-rational point, so it is sufficient to evaluate at that point instead. Finally, the map on Brauer groups induced by $S' \to Y$ preserves $\beta$ by Propositions \ref{galois_covering} and \ref{geometric_brauer} and therefore $ev_{\beta}: Y(\Q_2) \to \Z/2$ is also surjective. The surjectivity of the evaluation map at a single prime is sufficient to show the non-emptiness of the Brauer--Manin set. See the discussion of rational points on $Y$ in the proof of Theorem \ref{galoiscubicpoints} for the statement on the existence of Galois cubic points on $C$.     
\end{proof}

\section*{Acknowledgements}
I would like to express my gratitude to my supervisor Rachel Newton for introducing me to this topic and project, giving me careful supervision and proposing several helpful comments on the first draft of the paper.
I would also like to thank Alexei Skorobogatov for helpful suggestions, as well as Dario Beraldo, Federico Bongiorno and Riccardo Calini for their patience in our algebraic geometry discussions.  
Last but not least, I really appreciated the stimulating and relaxed research atmosphere in Rachel's group. Special thanks to Eric Zhu for carefully checking the first draft of the preprint. 
This work was supported by the Engineering and Physical Sciences Research Council [EP/S021590/1], the EPSRC Centre for Doctoral Training in Geometry and Number Theory (The London School of Geometry and Number Theory), University College London.

\bibliographystyle{alpha}

\bibliography{refs}

\end{document}